\documentclass[a4paper,11pt,leqno]{article}
\usepackage{latexsym, amsthm, amsmath, amsfonts}
\newtheorem{theorem}{Theorem}[section]
\newtheorem{lemma}[theorem]{Lemma}
\newtheorem{e-proposition}[theorem]{Proposition}

\newtheorem{corollary}[theorem]{Corollary}

\newtheorem{definition}[theorem]{Definition\rm}

\newtheorem{example}{\it Example\/}
\begin{document}
\title{Generic Riemannian submersions from nearly Kaehler manifolds}
\author{Rupali Kaushal, Rashmi Sachdeva, Rakesh Kumar\thanks{Corresponding Author} and R. K. Nagaich}
\date{}
\maketitle
\begin{abstract}
We study generic Riemannian submersions from nearly Kaehler manifolds onto Riemannian manifolds. We investigate conditions for the integrability of various distributions arising for generic Riemannian submersions and also obtain conditions for leaves to be totally geodesic foliations. We obtain conditions for a generic Riemannian submersion to be a totally geodesic map and also study generic Riemannian submersions with totally umbilical fibers. Finally, we derive conditions for generic Riemannian submersions to be harmonic map.
\end{abstract}
\section{Introduction}
Riemannian submersions between Riemannian manifolds equipped with an additional structure of almost complex type were introduced by Watson in \cite{Wat}. Watson defined an almost Hermitian submersion between almost Hermitian manifolds and showed that, in most of the cases, the base manifold and each fiber have the same kind of structure as the total space. Later, Sahin \cite{Sahin10} introduced the notion of anti-invariant Riemannian submersions from almost Hermitian manifolds onto Riemannian manifolds, where the vertical distribution is anti-invariant under the action of almost complex structure of the total manifold. As a generalization of anti-invariant submersions and almost Hermitian submersions, Sahin \cite{Sahin11} introduced the notion of semi-invariant Riemannian submersions from almost Hermitian manifolds onto Riemannian manifolds.\\
\indent As a generalization of semi-invariant submersions, Ali and Fatima~\cite{Ali} introduced the notion of generic Riemannian submersions from almost Hermitian manifolds onto Riemannian manifolds. Then, Fatima and Ali~\cite{Ali1} studied submersion of generic submanifolds of Kaehler manifolds onto almost Hermitian manifolds. Later, Akyol \cite{Akyol17} studied generic submersions from almost Riemannian product manifolds. Recently, Sayar et al. \cite{cem} introduced a new kind of Riemannian submersions, where the fibers are generic submanifolds, in the sense of Ronsse~\cite{Ronsse} and called such submersions as generic submersions.\\
\indent A more general and geometrically interesting class of almost Hermitian manifolds is of nearly Kaehler manifolds, which is one of the sixteen classes of almost Hermitian manifolds and given by Gray and Hervella in their celebrated paper \cite{grayharvella}. The geometrical meaning of nearly Kaehler condition is that the geodesics on the manifolds are holomorphically planar curves. Nearly Kaehler manifolds were extensively studied by Gray in \cite{Gray70} and a well known example of a non-Kaehlerian nearly Kaehler manifold is $6-$dimensional sphere.\\
\indent Ali and Fatima \cite{af13nk} studied anti-invariant Riemannian submersions from nearly Kaehler manifolds onto Riemannian manifolds. Recently, Rupali et al. \cite{rrrr} studied semi-invariant Riemannian submersions from nearly Kaehler manifolds onto Riemannian manifolds. In this paper, we study generic Riemannian submersions from nearly Kaehler manifolds onto Riemannian manifolds. We investigate conditions for the integrability of various distributions arising for generic Riemannian submersions and also obtained conditions for leaves to be totally geodesic foliations. We also obtain conditions for generic Riemannian submersion to be a totally geodesic map and to be a harmonic map.
\section{Preliminaries}
Let $(M, g_{M}, J)$ be an almost Hermitian manifold with a Riemannian metric $g_{M}$ and an almost complex structure $J$ such that
\begin{equation*}
J^{2} = -I, \quad g_{M}(X, Y) = g_{M}(JX, JY),
\end{equation*}
for any $X, Y\in\Gamma(TM)$. Let $\nabla$ be the Levi--Civita connection on $M$ with respect to $g_{M}$. If the almost complex structure $J$ is parallel with respect to $\nabla$, that is, $(\nabla_{X}J)Y = 0$, then $M$ is called a Kaehler manifold and if the tensor field $\nabla J$ is skew-symmetric, that is
\begin{equation}\label{eq:1}
(\nabla_{X}J)Y + (\nabla_{Y}J)X = 0,
\end{equation}
then $M$ is called a nearly Kaehler manifold.\\
\indent Let $(M, g_{M})$ and $(B, g_{B})$ be Riemannian manifolds of dimensions $m$ and $n$, respectively, where $m > n$. Then, a map $F: (M, g_{M})\rightarrow (B, g_{B})$ is called a Riemannian submersion \cite{neill} if it satisfies the following axioms:
\begin{itemize}
  \item[A1.] $F$ has maximal rank.
  \item[A2.] The differential map $F_{*}$ of $F$ preserves the scalar product of vectors normal to
the fibers.
\end{itemize}
Here, the fibers $F^{- 1}(y)$, $y \in B$ are $(m - n)-$dimensional submanifolds of $M$. A vector field on $M$ is called vertical (respectively, horizontal) if it is always tangent (respectively, orthogonal) to the fibers. The vertical distribution of $M$ is denoted by $\mathcal{V}_{p} = (ker F_{*p})$, $p\in M$, which is always integrable and the orthogonal distribution to $\mathcal{V}_{p}$ is denoted by $\mathcal{H}_{p} = (ker F_{*p})^{\bot}$, called the horizontal distribution, hence $TM = \mathcal{V} \oplus \mathcal{H}$. A vector field $X$ on $M$ is called a basic vector field if $X$ is horizontal and $F-$related to a vector field $X_{*}$ on $B$, that is, $F_{*}X_{p} = X_{*F(p)}$, for any $p\in M$.\\
\indent Next, we recall the following important lemma from O'Neill {\cite{neill}} for later uses.
\begin{lemma}\label{lem:1} Let $F: (M, g_{M})\rightarrow (B, g_{B})$ be a Riemannian submersion between Riemannian manifolds and let $\nabla$ and $\overline{\nabla}$ be the Levi--Civita connections of $M$ and $B$, respectively. If $X, Y$ are basic vector fields on $M$ and are $F-$related to $X_{*}$, $Y_{*}$ respectively, then
\begin{itemize}
  \item[(i)] $g_{M}(X, Y) = g_{B}(X_{*}, Y_{*})\circ F$.
  \item[(ii)] $\mathcal{H}[X, Y]$ is the basic vector field and $F-$related to $[X_{*}, Y_{*}]$.
  \item[(iii)] $(\nabla_{X}Y)^\mathcal{H}$ is the basic vector field and $F-$related to ${\overline{\nabla}}_{X_{*}}Y_{*}$.
  \item[(iv)] For any vertical vector field $V$, $[X, V]$ is always vertical.
\end{itemize}
\end{lemma}
\noindent It is known that the geometry of Riemannian submersions is characterized by O'Neill's tensors $\mathcal{T}$ and $\mathcal{A}$. For any arbitrary vector fields $U$ and $V$ on $M$, these tensors are defined as
\begin{equation}\label{eq:2}
\mathcal{T}_{U}V = \mathcal{H} \nabla_{\mathcal{V} U}\mathcal{V} V + \mathcal{V}\nabla_{\mathcal{V} U}\mathcal{H}V,\quad
\mathcal{A}_{U}V = \mathcal{V}\nabla_{\mathcal{H}U}\mathcal{H}V + \mathcal{H}\nabla_{\mathcal{H}U}\mathcal{V}V,
\end{equation}
where $\nabla$ is the Levi--Civita connection on $M$ with respect to $g_{M}$. It is easy to see that the tensor fields $\mathcal{T}$ and $\mathcal{A}$ are vertical and horizontal, respectively and satisfy
\begin{eqnarray}\label{eq:3}
\mathcal{T}_{U}V =  \mathcal{T}_{V}U, \quad\forall~ U, V\in\Gamma(ker F_{*}),
\end{eqnarray}
\begin{eqnarray}\label{eq:4}
\mathcal{A}_{X}Y = - \mathcal{A}_{Y}X = \frac{1}{2}\mathcal{V}[X, Y],\quad\forall~ X, Y\in\Gamma(ker F_{*})^\bot.
\end{eqnarray}
It should be noted that the tensor $\mathcal{T}$ serves as the second fundamental form of the fibers and hence a Riemannian submersion $F$ has totally geodesic fibers if and only if $\mathcal{T}$ vanishes identically. Moreover, relation in (\ref{eq:4}), shows that $\mathcal{A}$ is necessarily the integrability tensor of the horizontal distribution $(kerF_{*})^{\bot}$ on $M$. Using the definition of these tensors, we have the following important lemma from \cite{neill} immediately.
\begin{lemma}\label{lem:2} Let $X, Y$ be horizontal vector fields and $U, V$ be vertical vector fields. Then
\begin{itemize}
            \item [(i)] $\nabla_{U}V = \mathcal{T}_{U}V + \hat{\nabla}_{U}V$,
            \item [(ii)] $\nabla_{U}X = \mathcal H \nabla_{U}X + \mathcal{T}_{U}X$,
            \item [(iii)] $\nabla_{X}U = \mathcal{A}_{X}U + \mathcal{V} \nabla_{X}U$,
            \item [(iv)] $\nabla_{X}Y =\mathcal H \nabla_{X}Y + \mathcal{A}_{X}Y$,
\end{itemize}
where $\hat{\nabla}_{U}V = \mathcal{V}({\nabla}_{U}V).$ If $X$ is basic then $\mathcal{H}(\nabla_{U}X) = \mathcal{A}_{X}U$.
\end{lemma}
\noindent Let $F:(M, g_{M})\rightarrow (B, g_{B})$ be a smooth map between the Riemannian manifolds. Then, the differential $F_{*}$ of $F$ can be viewed as a section of the bundle $Hom(TM, F^{-1}(TB))$ of $M$, where $F^{-1}(TB)$ is the pullback bundle with fibres $(F^{-1}(TB))_p = T_{F(p)}B$, $p\in M$. $Hom(TM, F^{-1}(TB))$ has connection $\nabla^{F}$ induced from the Levi--Civita connection $\nabla$ of $M$ and a pull back connection. Then, the second fundamental form $(\nabla F_{*})$ of $F$ is given by
\begin{eqnarray}\label{eq:5}
(\nabla F_{*})(X,Y) = \nabla_{X}^{F} F_{*}(Y) - F_{*}(\nabla_{X}Y),
\end{eqnarray}
for any $X, Y\in\Gamma(TM)$. It should be noted that the second fundamental form is always symmetric. Further, the smooth map $F$ is said to be harmonic if $trace(\nabla F_{*})=0.$ The tension field $\tau(F)$ of $F$ is the section of $\Gamma(F^{-1}(TB))$ and given by
\begin{equation}\label{eq:6}
\tau(F) = div F_{*} = \sum^{m}_{i=1} (\nabla F_{*})(e_i, e_i),
\end{equation}
where $\{e_{1},\ldots, e_{m}\}$ is the orthonormal frame on $M$ then $F$ is harmonic if and only if $\tau(F)=0$, for more details, see \cite{Bairdwood}.
\section{Generic Riemannian Submersions}
Let $N$ be a real submanifold of an almost Hermitian manifold $(M, g_{M}, J)$ and let $\mathcal{D}_{p} = T_{p}M\cap JT_{p}M$, $p\in N$, be the maximal complex subspace of the tangent space $T_{p}M$ which is contained in $T_{p}N$. If the dimension of $\mathcal{D}_{p}$ is constant along $N$ and it defines a differentiable distribution on $N$, then $N$ is called a generic submanifold of $M$, for details, see \cite{chen}. Here, the distribution $\mathcal{D}_{p}$ on $N$ is called the holomorphic distribution. A generic submanifold is said to be a purely real submanifold if $\mathcal{D}_{p} = \{0\}$. Denote the orthogonal complementary distribution to $\mathcal{D}$ in $TN$ by $\mathcal{D}^{\bot}$, known as the purely real distribution and satisfies $\mathcal{D}_{p}\bot\mathcal{D}_{p}^{\bot}$, $\mathcal{D}_{p}^{\bot}\cap J\mathcal{D}_{p}^{\bot} = \{0\}$. Therefore, for any vector field $X$ tangent to $N$, we put $JX = tX + fX,$
where $tX$ and $fX$ are the tangential and normal components of $JX$, respectively. Further, for a generic submanifold, we have following important observations:
\begin{equation}\label{eq:7a}
t\mathcal{D} = \mathcal{D},\quad f\mathcal{D} = \{0\},\quad t\mathcal{D}_{p}^{\bot}\subset \mathcal{D}_{p}^{\bot},\quad f\mathcal{D}_{p}^{\bot}\subset T^{\bot}N.
\end{equation}
\begin{definition}\cite{Ali}.
Let $F: (M, g_{M}, J)\rightarrow (B, g_{B})$ be a Riemannian submersion from an almost Hermitian manifold onto a Riemannian manifold. Then the Riemannian submersion $F$ is called a generic Riemannian submersion if there is a distribution $\mathcal{D}\subset\Gamma(ker F_{*})$ such that
\begin{equation}\label{eq:7}
(ker F_{*}) = \mathcal{D}\oplus\mathcal{D}^{\bot},\quad J\mathcal{D} = \mathcal{D},
\end{equation}
where $\mathcal{D}^{\bot}$ is the orthogonal complementary of $\mathcal{D}$ in $(ker F_{*})$, and is called a purely real distribution on the fibers of the submersion $F$.
\end{definition}
\noindent It is well known that the vertical distribution $(ker F_{*})$ is always integrable. Hence, above definition implies that the integral manifolds (fibers) ${\pi}^{-1}(q)$, $q\in B$ of $(ker F_{*})$ are generic submanifolds of $M$.\\
\indent From the definition of generic Riemannian submersions, it is obvious that for any $U{\in{\Gamma(ker F_{*})}}$, we can write
\begin{equation}\label{eq:8}
JU = \phi{U} + \omega{U},
\end{equation}
where $\phi U\in\Gamma(ker F_{*})$ and $\omega U\in\Gamma(ker F_{*})^{\bot}$. From (\ref{eq:7a}), it is clear that $\omega\mathcal{D} = \{0\}$ and $\omega\mathcal{D}^{\bot}\in\Gamma(ker F_{*})^{\bot}$. Denote the complementary distribution to $\omega\mathcal{D}^{\bot}$ in $(ker F_{*})^{\bot}$ by $\mu$ then we get $(ker F_{*})^{\bot} = \omega\mathcal{D}^{\bot}\oplus\mu$, and that $\mu$ is invariant under $J$. Thus, for $X\in\Gamma(ker F_{*})^{\bot}$, we can write
\begin{equation}\label{eq:9}
JX = \mathcal{B}X + \mathcal{C}X,
\end{equation}
where $\mathcal{B}X\in\Gamma(\mathcal{D}^{\bot})$ and $\mathcal{C}X\in\Gamma(\mu)$.
\begin{example}
Let $(\mathbb{R}^{8}, J, g_{1})$ be an almost Hermitian manifold endowed with an almost complex structure $(J, g_{1})$ and given by $g_{1} = dx_{1}^{2} + dx_{2}^{2} + dx_{3}^{2} + dx_{4}^{2} + dx_{5}^{2} + dx_{6}^{2} + dx_{7}^{2} + dx_{8}^{2}$, $J(x_{1}, x_{2}, x_{3}, x_{4}, x_{5}, x_{6}, x_{7}, x_{8}) = (- x_{2}, x_{1}, - x_{4}, x_{3}, - x_{6},\\ x_{5}, - x_{8}, x_{7}).$ Let $(\mathbb{R}^{4}, g_{2})$ be a Riemannian manifold endowed with metric $g_{2} = dy_{1}^{2} + dy_{2}^{2} + dy_{3}^{2} + dy_{4}^{2}$ and $F:(\mathbb{R}^{8}, J, g_{1})\rightarrow (\mathbb{R}^{4}, g_{2})$ be a map defined by
$F(x_{1}, x_{2}, x_{3}, x_{4}, x_{5}, x_{6}, x_{7}, x_{8}) = (\frac{x_{1} + x_{3}}{\sqrt{2}}, \frac{x_{2} + x_{4}}{\sqrt{2}}, \sin\alpha x_{5} + \cos\alpha x_{7}, \sin\alpha x_{6} - \cos\alpha x_{8}).$
Then, by straightforward calculations
\begin{eqnarray*}
(ker F_{*})& =& span\Big\{X_{1} = \partial x_{1} - \partial x_{3}, X_{2} = \partial x_{2} - \partial x_{4},\nonumber\\&&
X_{3} = \cos\alpha\partial x_{5} - \sin\alpha\partial x_{7}, X_{4} = \cos\alpha\partial x_{6} + \sin\alpha\partial x_{8}\Big\}.
\end{eqnarray*}
Clearly, $JX_{1} = X_{2}$ therefore $\mathcal{D} = span\{X_{1}, X_{2}\}$ and $\mathcal{D}^{\bot}$ is a slant distribution with slant angle $2\alpha$ therefore it is a purely real distribution. Moreover
\begin{eqnarray*}
(ker F_{*})^{\bot}& =& span\Big\{Z_{1} = \partial x_{1} + \partial x_{3}, Z_{2} = \partial x_{2} + \partial x_{4},\nonumber\\&& Z_{3} = \sin\alpha\partial x_{5} + \cos\alpha\partial x_{7}, Z_{4} = \sin\alpha\partial x_{6} - \cos\alpha\partial x_{8}\Big\}.
\end{eqnarray*}
Since $JZ_{1} = - Z_{2}$ therefore $\mu = span\{Z_{1}, Z_{2}\}$ and invariant with respect to $J$. Furthermore, we derive $F_{*}Z_{1} = \sqrt{2}\partial y_{1}$, $F_{*}Z_{2} = \sqrt{2}\partial y_{2}$, $F_{*}Z_{3} = \partial y_{3}$ and $F_{*}Z_{4} = \partial y_{4}$ such that
$$
g_{1}(Z_{1}, Z_{1}) = 2 = g_{2}(F_{*}Z_{1}, F_{*}Z_{1}),\quad g_{1}(Z_{2}, Z_{2}) = 2 = g_{2}(F_{*}Z_{2}, F_{*}Z_{2}),
$$
$$
g_{1}(Z_{3}, Z_{3}) = 1 = g_{2}(F_{*}Z_{3}, F_{*}Z_{3}),\quad g_{1}(Z_{4}, Z_{4}) = 1 = g_{2}(F_{*}Z_{4}, F_{*}Z_{4}).
$$
Hence, $F$ is a proper generic Riemannian submersion from an almost Hermitian manifold onto a Riemannian manifold.
\end{example}
\begin{example}
Let $F:(\mathbb{R}^{8}, J, g_{1})\rightarrow (\mathbb{R}^{4}, g_{2})$ be a map defined by
$$
F(x_{1}, x_{2}, x_{3}, x_{4}, x_{5}, x_{6}, x_{7}, x_{8}) = \Big(\frac{x_{2} - x_{3}}{\sqrt{2}}, x_{4}, x_{5}, x_{6}\Big),
$$
where $(\mathbb{R}^{8}, J, g_{1})$ and $(\mathbb{R}^{4}, g_{2})$ defined as in above example. Then, by straightforward calculations
\begin{eqnarray*}
(ker F_{*}) = span\Big\{X_{1} = \partial x_{1}, X_{2} = \partial x_{2} + \partial x_{3},
X_{3} = \partial x_{7}, X_{4} = \partial x_{8}\Big\}.
\end{eqnarray*}
Clearly, $\mathcal{D} = span\{X_{3}, X_{4}\}$ and $\mathcal{D}^{\bot}$ is a slant distribution with slant angle $\frac{\pi}{4}$ therefore it is a purely real distribution. Moreover
\begin{eqnarray*}
(ker F_{*})^{\bot} = span\Big\{Z_{1} = \partial x_{2} - \partial x_{3}, Z_{2} = \partial x_{4}, Z_{3} = \partial x_{5}, Z_{4} = \partial x_{6}\Big\}.
\end{eqnarray*}
where $\mu = span\{Z_{3}, Z_{4}\}$. Furthermore, we derive $F_{*}Z_{1} = \sqrt{2}\partial y_{1}$, $F_{*}Z_{2} = \partial y_{2}$, $F_{*}Z_{3} = \partial y_{3}$ and $F_{*}Z_{4} = \partial y_{4}$ such that
$$
g_{1}(Z_{1}, Z_{1}) = 2 = g_{2}(F_{*}Z_{1}, F_{*}Z_{1}),\quad g_{1}(Z_{2}, Z_{2}) = 1 = g_{2}(F_{*}Z_{2}, F_{*}Z_{2}),
$$
$$
g_{1}(Z_{3}, Z_{3}) = 1 = g_{2}(F_{*}Z_{3}, F_{*}Z_{3}),\quad g_{1}(Z_{4}, Z_{4}) = 1 = g_{2}(F_{*}Z_{4}, F_{*}Z_{4}).
$$
Hence, $F$ is a proper generic Riemannian submersion from an almost Hermitian manifold onto a Riemannian manifold.
\end{example}
\begin{example}
Every semi-invariant submersion \cite{Sahin11} is a generic Riemannian submersion with a totally real distribution $\mathcal{D}^{\bot}$.
\end{example}
\begin{example}
Every slant submersion \cite{Sahin11a} is a generic Riemannian submersion with $\mathcal{D} = \{0\}$ and slant distribution $\mathcal{D}^{\bot}$.
\end{example}
\begin{example}
Every semi-slant submersion \cite{park} is a generic Riemannian submersion with slant distribution $\mathcal{D}^{\bot}$.
\end{example}
\noindent Next, from (\ref{eq:7a})--(\ref{eq:9}), we get the following lemma immediately for later use.
\begin{lemma}\label{lem:4}
Let $F: (M, g_{M}, J)\rightarrow (B, g_{B})$ be a generic Riemannian submersion from an almost Hermitian manifold onto a Riemannian manifold. Then
\begin{itemize}
\item[(i)] $\phi\mathcal{D} = \mathcal{D}$,\quad $\phi\mathcal{D}^{\bot}\subset\mathcal{D}^{\bot}$,\quad $\mathcal{B}(ker F_{*})^{\bot} = \mathcal{D}^{\bot}$.
\item[(ii)] $\phi^{2} + \mathcal{B}\omega = - id$,\quad $\mathcal{C}^{2} + \omega\mathcal{B} = -id$,\quad $\omega\phi + \mathcal{C}\omega = 0$,\quad $\mathcal{B}\mathcal{C} + \phi\mathcal{B} = 0$.
\end{itemize}
\end{lemma}
\noindent For any arbitrary tangent vector fields $U$ and $V$ on $M$, we set
\begin{eqnarray}\label{eq:10}
(\nabla_{U}J)V = \mathcal{P}_{U}V + \mathcal{Q}_{U}V,
\end{eqnarray}
where $\mathcal{P}_{U}V$ (respectively, $\mathcal{Q}_{U}V$) denotes the horizontal (respectively, vertical) part of $(\nabla_{U}J)V$. Clearly, if $M$ is a Kaehler manifold then $\mathcal{P} = \mathcal{Q} = 0$. and if $M$ is a nearly Kaehler manifold then $\mathcal{P}$ and $\mathcal{Q}$ satisfy
\begin{eqnarray}\label{eq:11}
\mathcal{P}_{U}V = - \mathcal{P}_{V}U,\quad \mathcal{Q}_{U}V = - \mathcal{Q}_{V}U.
\end{eqnarray}
Hence, from the Lemma~\ref{lem:2} and (\ref{eq:8}), (\ref{eq:9}), we derive covariant derivative of $\phi$, $\omega$, $\mathcal{B}$ and $\mathcal{C}$ as below:
\begin{theorem}
Let $F: (M, g_{M}, J)\rightarrow (B, g_{B})$ be a generic Riemannian submersion from a nearly Kaehler manifold onto a Riemannian manifold. Then
\begin{equation}\label{eq:11c}
(\nabla_{U}\phi)V = \mathcal{B}\mathcal{T}_{U}V - \mathcal{T}_{U}\omega V + \mathcal{Q}_{U}V,
\end{equation}
\begin{equation}\label{eq:11a}
(\nabla_{U}\omega)V = \mathcal{C}\mathcal{T}_{U}V - \mathcal{T}_{U}\phi V + \mathcal{P}_{U}V,
\end{equation}
\begin{equation*}
(\nabla_{U}\mathcal{B})X = \phi\mathcal{T}_{U}X - \mathcal{T}_{U}\mathcal{C}X + \mathcal{Q}_{U}X,
\end{equation*}
\begin{equation*}
(\nabla_{U}\mathcal{C})X = \omega\mathcal{T}_{U}X - \mathcal{T}_{U}\mathcal{B}X + \mathcal{P}_{U}X,
\end{equation*}
where
\begin{equation}\label{eq:11b}
(\nabla_{U} \phi)V ={\hat{\nabla}}_{U} \phi V - \phi {\hat{\nabla}}_{U}V,\quad (\nabla_{U}\omega)V =\mathcal{H}{\nabla}_{U}\omega V - \omega \hat{\nabla}_{U}V,
\end{equation}
\begin{equation*}
(\nabla_{U}\mathcal{B})X = \mathcal{V}\nabla_{U}\mathcal{B}X - \mathcal{B}\mathcal{H}\nabla_{U}X,\quad (\nabla_{U}\mathcal{C})X = \mathcal{H}\nabla_{U}\mathcal{C}X - \mathcal{C}\mathcal{H}\nabla_{U}X,
\end{equation*}
for any $U,V \in\Gamma(ker F_{*})$ and $X\in\Gamma(ker F_{*})^{\bot}$.
\end{theorem}
\begin{lemma}
Let $F: (M, g_{M}, J)\rightarrow (B, g_{B})$ be a generic Riemannian submersion from a nearly Kaehler manifold onto a Riemannian manifold. Then
\begin{eqnarray}\label{eq:12}
\left\{
\begin{array}{c}
  \mathcal{B}\mathcal{T}_{U}V + \phi\mathcal{V}\nabla_{U}V  = \mathcal{V}\nabla_{U}\phi{V} + \mathcal{T}_{U}\omega{V} + \mathcal{Q}_{V}U,\\
\mathcal{C}\mathcal{T}_{U}V + \omega\mathcal{V}\nabla_{U}V = \mathcal{T}_{U}\phi V + \mathcal{H}\nabla_{U}\omega V + \mathcal{P}_{V}U,
\end{array}
\right.
\end{eqnarray}
\begin{eqnarray}\label{eq:13}
\left\{
\begin{array}{c}
\mathcal{B}\mathcal{H}\nabla_{X}Y + \phi\mathcal{A}_{X}Y  = \mathcal{V}\nabla_{X}\mathcal{B}Y + \mathcal{A}_{X}\mathcal{C}Y + \mathcal{Q}_{Y}X,\\
\mathcal{C}\mathcal{H}\nabla_{X}Y + \omega\mathcal{A}_{X}Y = \mathcal{A}_{X}\mathcal{B}Y + \mathcal{H}\nabla_{X}\mathcal{C}Y + \mathcal{P}_{Y}X,
\end{array}
\right.
\end{eqnarray}
\begin{eqnarray}\label{eq:14}
\left\{
\begin{array}{c}
\mathcal{B}\mathcal{A}_{X}V + \phi\mathcal{V}\nabla_{X}V  = \mathcal{V}\nabla_{X}\phi{V} + \mathcal{A}_{X}\omega{V} + \mathcal{Q}_{V}X,\\
\mathcal{C}\mathcal{A}_{X}V + \omega\mathcal{V}\nabla_{X}V = \mathcal{A}_{X}\phi V + \mathcal{H}\nabla_{X}\omega V + \mathcal{P}_{V}X,
\end{array}
\right.
\end{eqnarray}
for any $X, Y\in\Gamma(ker F_{*})^{\perp}$ and $U, V\in\Gamma(ker F_{*})$.
\end{lemma}
\begin{proof}
Let $U, V\in\Gamma(ker F_{*})$ then from (\ref{eq:1}), we have $J\nabla_{U}V = \nabla_{U}JV + (\nabla_{V}J)U$. Using the Lemma~\ref{lem:2} with (\ref{eq:8})--(\ref{eq:10}), we get
\begin{eqnarray*}
&&\mathcal{B}\mathcal{T}_{U}V + \mathcal{C}\mathcal{T}_{U}V + \phi\mathcal{V}\nabla_{U}V + \omega\mathcal{V}\nabla_{U}V = \mathcal{T}_{U}\phi{V} + \mathcal{V}\nabla_{U}\phi{V}\nonumber\\&&
+ \mathcal{H}(\nabla_{U}\omega{V}) + \mathcal{T}_{U}\omega{V} + \mathcal{P}_{V}U + \mathcal{Q}_{V}U.
\end{eqnarray*}
On comparing the vertical and horizontal parts of the last expression, we obtain (\ref{eq:12}). Analogously, for $X, Y\in\Gamma(ker F_{*})^{\bot}$ and $U\in\Gamma(ker F_{*})$, we derive (\ref{eq:13}) and (\ref{eq:14}).
\end{proof}
\begin{corollary}
Take $U\in\Gamma(ker F_{*})$, $V\in\Gamma(\mathcal{D})$ and $\xi\in\Gamma(\mu)$ in $(\ref{eq:12})_{2}$ then it follows that
$$
g_{M}(\mathcal{C}\mathcal{T}_{U}V, \xi) = g_{M}(\mathcal{T}_{U}\phi V, \xi) + g_{M}(\mathcal{P}_{V}U,  \xi),
$$
this, further implies that
$$
g_{M}((\mathcal{T}_{U}J)V, \xi) = g_{M}(\mathcal{P}_{U}V,  \xi).
$$
Similarly, for $X\in\Gamma(ker F_{*})^{\bot}$, $Y\in\Gamma(\mu)$ and $U\in\Gamma(\mathcal{D})$, from $(\ref{eq:13})_{1}$, we derive
$$
g_{M}((\mathcal{A}_{X}J)Y, U) = g_{M}(\mathcal{Q}_{X}Y, U).
$$
For $X\in\Gamma(ker F_{*})^{\bot}$, $Y, Z\in\Gamma(\omega\mathcal{D}^{\bot})$, from $(\ref{eq:13})_{2}$, we obtain
$$
g_{M}((\mathcal{A}_{X}J)Y, Z) = g_{M}(\mathcal{P}_{X}Y, Z).
$$
For $X\in\Gamma(ker F_{*})^{\bot}$, $Y\in\Gamma(\mu)$ and $V\in\Gamma(\mathcal{D})$, from $(\ref{eq:14})_{2}$, we get
$$
g_{M}((\mathcal{A}_{X}J)V, Y) = g_{M}(\mathcal{P}_{X}V, Y).
$$
\end{corollary}
\begin{theorem}\label{thm:1}
Let $F: (M, g_{M}, J)\rightarrow (B, g_{B})$ be a generic Riemannian submersion from a nearly Kaehler manifold onto a Riemannian manifold. Then the distribution $\mathcal{D}$ defines a totally geodesic foliation in $M$ if and only if
\begin{eqnarray*}
\mathcal{V}\nabla_{U}\phi V - \mathcal{Q}_{U}V\in\Gamma(\mathcal{D}),\quad \mathcal{T}_{U}\phi V - \mathcal{P}_{U}V\in\Gamma(\mu),
\end{eqnarray*}
for any $U, V\in\Gamma(\mathcal{D})$.
\end{theorem}
\begin{proof}
Let $U, V\in\Gamma(\mathcal{D})$ then using the Lemma~\ref{lem:2} with (\ref{eq:1}) and (\ref{eq:10}), we derive
$$
\nabla_{U}V = - J(\mathcal{T}_{U}\phi V + \mathcal{V}\nabla_{U}\phi V + \mathcal{P}_{V}U + \mathcal{Q}_{V}U).
$$
Take scalar product of last expression with $W\in\Gamma(\mathcal{D}^{\bot})$, we get
\begin{equation}\label{eq:15}
g_{M}(\nabla_{U}V, W) = g_{M}(\mathcal{V}\nabla_{U}\phi V - \mathcal{Q}_{U}V, \phi W) + g_{M}(\mathcal{T}_{U}\phi V - \mathcal{P}_{U}V, \omega W),
\end{equation}
hence the proof is complete.
\end{proof}
On interchanging the role of $U$, $V$ and subtracting the resulting expression form (\ref{eq:15}) and using (\ref{eq:11}), we obtain the following observation immediately.
\begin{corollary}
Let $F: (M, g_{M}, J)\rightarrow (B, g_{B})$ be a generic Riemannian submersion from a nearly Kaehler manifold onto a Riemannian manifold. Then the distribution $\mathcal{D}$ is integrable if and only if
\begin{equation*}
\mathcal{V}\nabla_{U}\phi V - \mathcal{V}\nabla_{V}\phi U - 2\mathcal{Q}_{U}V\in\Gamma(\mathcal{D}),\quad \mathcal{T}_{U}\phi V - \mathcal{T}_{V}\phi U - 2\mathcal{P}_{U}V\in\Gamma(\mu),
\end{equation*}
for any $U, V\in\Gamma(\mathcal{D})$.
\end{corollary}
\begin{theorem}\label{thm:2}
Let $F: (M, g_{M}, J)\rightarrow (B, g_{B})$ be a generic Riemannian submersion from a nearly Kaehler manifold onto a Riemannian manifold. Then the distribution $\mathcal{D}^{\bot}$ defines a totally geodesic foliation in $M$ if and only if
\begin{equation*}
\mathcal{V}\nabla_{U}\phi V + \mathcal{T}_{U}\omega V - \mathcal{Q}_{U}V\in\Gamma(\mathcal{D}^{\bot}),
\end{equation*}
for any $U, V\in\Gamma(\mathcal{D}^{\bot})$.
\end{theorem}
\begin{proof}
Let $U, V\in\Gamma(\mathcal{D}^{\bot})$ then using the Lemma~\ref{lem:2} with (\ref{eq:1}) and (\ref{eq:10}), we derive
$$
\nabla_{U}V = - J(\mathcal{T}_{U}\phi V + \mathcal{V}\nabla_{U}\phi V + \mathcal{H}\nabla_{U}\omega V + \mathcal{T}_{U}\omega V + \mathcal{P}_{V}U + \mathcal{Q}_{V}U).
$$
Take scalar product of last expression with $W\in\Gamma(\mathcal{D})$, we get
\begin{equation*}
g_{M}(\nabla_{U}V, W) = g_{M}(\mathcal{V}\nabla_{U}\phi V + \mathcal{T}_{U}\omega V - \mathcal{Q}_{U}V, \phi W),
\end{equation*}
hence the proof is complete.
\end{proof}
\begin{corollary}
Let $F: (M, g_{M}, J)\rightarrow (B, g_{B})$ be a generic Riemannian submersion from a nearly Kaehler manifold onto a Riemannian manifold. Then the distribution $\mathcal{D}^{\bot}$ is integrable if and only if
\begin{equation*}
\mathcal{V}\nabla_{U}\phi V - \mathcal{V}\nabla_{V}\phi U + \mathcal{T}_{U}\omega V - \mathcal{T}_{V}\omega U - 2\mathcal{Q}_{U}V\in\Gamma(\mathcal{D}^{\bot}),
\end{equation*}
for any $U, V\in\Gamma(\mathcal{D}^{\bot})$.
\end{corollary}
Now, we recall an important theorem for product structures from \cite{ponge}.
\begin{theorem}\label{thm:A}
Let $g$ be a Riemannian metric tensor on the manifold $\overline{M} = M\times N$ and assume that the canonical foliations $\mathcal{D}_{M}$ and $\mathcal{D}_{N}$ intersect perpendicularly everywhere. Then $g$ is the metric tensor of:
\begin{itemize}
  \item[(i)] a double-twisted product $M\times_{(f, g)}N$ if and only if $\mathcal{D}_{M}$ and $\mathcal{D}_{N}$ are totally umbilical foliations,
  \item[(ii)] a twisted product $M\times_{f}N$ if and only if $\mathcal{D}_{M}$ is a totally geodesic foliation and $\mathcal{D}_{N}$ is a totally umbilical foliation,
  \item[(iii)] a warped product $M\times_{f}N$ if and only if $\mathcal{D}_{M}$ is a totally geodesic foliation and $\mathcal{D}_{N}$ is a spherical foliation, i.e., it is umbilical and its mean curvature vector field is
parallel, and
  \item[(iv)] a usual product of Riemannian manifolds if and only if $\mathcal{D}_{M}$ and $\mathcal{D}_{N}$ are totally
geodesic foliations.
\end{itemize}
\end{theorem}
\noindent Thus, from the Theorems \ref{thm:1} and \ref{thm:2}, we have the following assertion immediately.
\begin{theorem}\label{thm:3}
Let $F: (M, g_{M}, J)\rightarrow (B, g_{B})$ be a generic Riemannian submersion from a nearly Kaehler manifold onto a Riemannian manifold. Then fibers of $F$ are locally product Riemannian manifolds of the form $M_{\mathcal{D}}\times M_{\mathcal{D}^{\bot}}$, where $M_{\mathcal{D}}$ and $M_{\mathcal{D}^{\bot}}$ are integral manifolds of the vertical distribution $(ker F_{*})$ if and only if
\begin{eqnarray}\label{eq:16}
\mathcal{V}\nabla_{U}\phi V - \mathcal{Q}_{U}V\in\Gamma(\mathcal{D}),\quad \mathcal{T}_{U}\phi V - \mathcal{P}_{U}V\in\Gamma(\mu),
\end{eqnarray}
and
\begin{equation}\label{eq:17}
\mathcal{V}\nabla_{W}\phi Z + \mathcal{T}_{W}\omega Z - \mathcal{Q}_{W}Z\in\Gamma(\mathcal{D}^{\bot}),
\end{equation}
for any $U, V\in\Gamma(\mathcal{D})$, $W, Z\in\Gamma(\mathcal{D}^{\bot})$.
\end{theorem}
\begin{theorem}\label{thm:4}
Let $F: (M, g_{M}, J)\rightarrow (B, g_{B})$ be a generic Riemannian submersion from a nearly Kaehler manifold onto a Riemannian manifold. Then the horizontal distribution $(ker F_{*})^{\bot}$ defines a totally geodesic foliation in $M$ if and only if
\begin{equation}\label{eq:18}
\mathcal{A}_{X}\mathcal{B}Y + \mathcal{H}\nabla_{X}\mathcal{C}Y - \mathcal{P}_{X}Y\in\Gamma(\mu),\quad \mathcal{V}\nabla_{X}\mathcal{B}Y + \mathcal{A}_{X}\mathcal{C}Y - \mathcal{Q}_{X}Y = 0,
\end{equation}
for any $X, Y\in\Gamma(ker F_{*})^{\bot}$.
\end{theorem}
\begin{proof}
For any for any $X, Y\in\Gamma(ker F_{*})^{\bot}$, we have
$$
\nabla_{X}Y = - J(\mathcal{A}_{X}\mathcal{B}Y + \mathcal{V}\nabla_{X}\mathcal{B}Y + \mathcal{H}\nabla_{X}\mathcal{C}Y + \mathcal{A}_{X}\mathcal{C}Y + \mathcal{P}_{Y}X + \mathcal{Q}_{Y}X).
$$
On taking scalar product of last expression with $U\in\Gamma(ker F_{*})$, the proof is complete.
\end{proof}
\indent It is known that for a Riemannian submersion, the vertical distribution $(ker F_{*})$ is always integrable and the horizontal distribution $(ker F_{*})$ is not always integrable. Therefore, from the above Theorem~\ref{thm:4}, we obtain necessary and sufficient conditions for the horizontal distribution to be integrable.
\begin{corollary}
Let $F: (M, g_{M}, J)\rightarrow (B, g_{B})$ be a generic Riemannian submersion from a nearly Kaehler manifold onto a Riemannian manifold. Then, the horizontal distribution $(ker F_{*})^{\bot}$ is integrable if and only if
\begin{eqnarray*}
&&\mathcal{A}_{X}\mathcal{B}Y - \mathcal{A}_{Y}\mathcal{B}X + \mathcal{H}\nabla_{X}\mathcal{C}Y - \mathcal{H}\nabla_{Y}\mathcal{C}X - 2\mathcal{P}_{X}Y\in\Gamma(\mu),\nonumber\\&&
\mathcal{V}\nabla_{X}\mathcal{B}Y - \mathcal{V}\nabla_{Y}\mathcal{B}X + \mathcal{A}_{X}\mathcal{C}Y - \mathcal{A}_{Y}\mathcal{C}X - 2\mathcal{Q}_{X}Y = 0,
\end{eqnarray*}
for any $X, Y\in\Gamma(ker F_{*})^{\bot}$.
\end{corollary}
\begin{theorem}\label{thm:5}
Let $F: (M, g_{M}, J)\rightarrow (B, g_{B})$ be a generic Riemannian submersion from a nearly Kaehler manifold onto a Riemannian manifold. Then, the vertical distribution $(ker F_{*})$ defines a totally geodesic foliation in $M$ if and only if
\begin{eqnarray}\label{eq:19}
\mathcal{T}_{U}\phi V + \mathcal{H}\nabla_{U}\omega V - \mathcal{P}_{U}V\in\Gamma(\omega\mathcal{D}^{\bot}),
\end{eqnarray}
\begin{eqnarray}\label{eq:20}
\mathcal{V}\nabla_{U}\phi V + \mathcal{T}_{U}\omega V - \mathcal{Q}_{U}V\in\Gamma(\mathcal{D}),
\end{eqnarray}
for any $U, V\in\Gamma(ker F_{*})$.
\end{theorem}
\begin{proof}
Let $U, V\in\Gamma(ker F_{*})$, then using the Lemma~\ref{lem:2} with (\ref{eq:1}) and (\ref{eq:10}), we obtain
\begin{eqnarray*}
\nabla_{U}V& =& - J(\mathcal{T}_{U}\phi V + \mathcal{V}\nabla_{U}\phi V + \mathcal{H}\nabla_{U}\omega V + \mathcal{T}_{U}\omega V + \mathcal{P}_{V}U + \mathcal{Q}_{V}U)\nonumber\\&&
= - \{\mathcal{B}(\mathcal{T}_{U}\phi V + \mathcal{H}\nabla_{U}\omega V + \mathcal{P}_{V}U) + \phi(\mathcal{V}\nabla_{U}\phi V  + \mathcal{T}_{U}\omega V + \mathcal{Q}_{V}U)\}\nonumber\\&&\quad
- \{\mathcal{C}(\mathcal{T}_{U}\phi V + \mathcal{H}\nabla_{U}\omega V + \mathcal{P}_{V}U) + \omega(\mathcal{V}\nabla_{U}\phi V
+ \mathcal{T}_{U}\omega V + \mathcal{Q}_{V}U)\}.
\end{eqnarray*}
Hence, the vertical distribution $(ker F_{*})$ defines a totally geodesic foliation in $M$ if and only if
\begin{equation*}
\mathcal{C}(\mathcal{T}_{U}\phi V + \mathcal{H}\nabla_{U}\omega V + \mathcal{P}_{V}U) = 0,\quad
\omega(\mathcal{V}\nabla_{U}\phi V + \mathcal{T}_{U}\omega V + \mathcal{Q}_{V}U) = 0,
\end{equation*}
thus, the proof is complete.
\end{proof}
\noindent Thus, from the Theorems \ref{thm:1}, \ref{thm:2}, \ref{thm:4} and \ref{thm:5}, we have following observations:
\begin{corollary}
Let $F: (M, g_{M}, J)\rightarrow (B, g_{B})$ be a generic Riemannian submersion from a nearly Kaehler manifold onto a Riemannian manifold. Then, the total space $M$ of generic Riemannian submersion $F$ is a locally product Riemannian manifold of the form $M_{\mathcal{D}}\times M_{\mathcal{D}^{\bot}}\times M_{(ker F_{*})^{\bot}}$ if and only if the expressions in (\ref{eq:16})--(\ref{eq:18}) hold, where $M_{\mathcal{D}}$, $M_{\mathcal{D}^{\bot}}$ and $M_{(ker F_{*})^{\bot}}$ are integral manifolds of $(ker F_{*})$ and $(ker F_{*})^{\bot}$, respectively
\end{corollary}
\begin{corollary}
Let $F: (M, g_{M}, J)\rightarrow (B, g_{B})$ be a generic Riemannian submersion from a nearly Kaehler manifold onto a Riemannian manifold. Then, the total space $M$ of generic Riemannian submersion $F$ is a locally product Riemannian manifold of the form $M_{(ker F_{*})}\times M_{(ker F_{*})^{\bot}}$ if and only if the expressions in (\ref{eq:18})--(\ref{eq:20}) hold, where $M_{(ker F_{*})}$ and $M_{(ker F_{*})^{\bot}}$ are integral manifolds of $(ker F_{*})$ and $(ker F_{*})^{\bot}$, respectively.
\end{corollary}
Let $F:(M, g_{M})\rightarrow (B, g_{B})$ be a smooth map between Riemannian manifolds. If $F$ maps every geodesic in the total manifold into a geodesic in the base manifold, in proportion to the arc length then the mapping $F$ is called a totally geodesic map. In other words, $F$ is called a totally geodesic map if and only if the second fundamental form of the map $F$ vanishes identically, that is, $(\nabla F_{*}) = 0$.
\begin{theorem}
Let $F$ be a generic Riemannian submersion from a nearly Kaehler manifold $(M, g, J)$ to a Riemannian manifold ${(B, g_{B})}$. Then, the map $F$ is a totally geodesic map if and only if
\begin{eqnarray*}
(\mathcal{T}_{U}\phi V + \mathcal{H}\nabla_{U}\omega V - \mathcal{P}_{U}V), (\mathcal{T}_{U}\mathcal{B}X + \mathcal{H}\nabla_{U}\mathcal{C}X - \mathcal{P}_{U}X)~ \textrm{belong to}~\Gamma(\omega\mathcal{D}^{\bot}),
\end{eqnarray*}
and
\begin{eqnarray*}
(\mathcal{V}\nabla_{U}\phi V + \mathcal{T}_{U}\omega V - \mathcal{Q}_{U}V), (\mathcal{V}\nabla_{U}\mathcal{B}X + \mathcal{T}_{U}\mathcal{C}X - \mathcal{Q}_{U}X)~ \textrm{belong to}~ \Gamma(\mathcal{D}),
\end{eqnarray*}
for any $U, V\in\Gamma(ker F_{*})$ and $X, Y\in\Gamma(ker F_{*})^{\bot}$.
\end{theorem}
\begin{proof}
From the Lemma \ref{lem:1} and (\ref{eq:5}), it is obvious that the second fundamental form $(\nabla F_{*})$ of the Riemannian submersion $F$ satisfies $(\nabla F_{*})(X, Y) = 0$, for any $X, Y\in\Gamma(kerF_{*})^{\bot}$. Therefore, the generic Riemannian submersion $F$ is a totally geodesic map if and only if $(\nabla F_{*})(U, V) = 0$ and $(\nabla F_{*})(U, X) = 0$, for any $U, V\in\Gamma(kerF_{*})$ and $X\in\Gamma(kerF_{*})^{\bot}$. From the Lemma~\ref{lem:1}, (\ref{eq:1}) and (\ref{eq:5}), we can write
\begin{eqnarray}\label{eq:20a}
&(\nabla F_{*})(U, V)& = - F_{*}(\nabla_{U}V) = F_{*}(J^{2}\nabla_{U}V)\nonumber\\&&
= F_{*}\big\{\mathcal{B}\mathcal{T}_{U}\phi V + \mathcal{C}\mathcal{T}_{U}\phi V + \phi\mathcal{V}\nabla_{U}\phi V + \omega\mathcal{V}\nabla_{U}\phi V\nonumber\\&&\quad + \mathcal{B}\mathcal{H}\nabla_{U}\omega V + \mathcal{C}\mathcal{H}\nabla_{U}\omega V + \phi\mathcal{T}_{U}\omega V + \omega\mathcal{T}_{U}\omega V\nonumber\\&&\quad
+ \mathcal{B}\mathcal{P}_{V}U + \mathcal{C}\mathcal{P}_{V}U + \phi\mathcal{Q}_{V}U + \omega\mathcal{Q}_{V}U\big\}\nonumber\\&&
= F_{*}\big\{\mathcal{C}\mathcal{T}_{U}\phi V + \omega\mathcal{V}\nabla_{U}\phi V + \mathcal{C}\mathcal{H}\nabla_{U}\omega V + \omega\mathcal{T}_{U}\omega V\nonumber\\&&\quad - \mathcal{C}\mathcal{P}_{U}V - \omega\mathcal{Q}_{U}V\big\}.
\end{eqnarray}
Hence, $(\nabla F_{*})(U, V) = 0$ if and only if $\mathcal{C}\{\mathcal{T}_{U}\phi V + \mathcal{H}\nabla_{U}\omega V - \mathcal{P}_{U}V\} = 0$ and $\omega\{\mathcal{V}\nabla_{U}\phi V + \mathcal{T}_{U}\omega V - \mathcal{Q}_{U}V\} = 0$, that is, if and only if $(\mathcal{T}_{U}\phi V + \mathcal{H}\nabla_{U}\omega V - \mathcal{P}_{U}V)\in\Gamma(\omega\mathcal{D}^{\bot})$ and $(\mathcal{V}\nabla_{U}\phi V + \mathcal{T}_{U}\omega V - \mathcal{Q}_{U}V)\in\Gamma(\mathcal{D})$.\\
\indent It is known that the second fundamental form of the map is symmetric then analogous to above derivation, for $U\in\Gamma(ker F{*})$ and $X\in\Gamma(ker F_{*})^{\bot}$, we obtain
\begin{eqnarray*}
(\nabla F_{*})(U, X)& =& F_{*}\big\{\mathcal{C}\mathcal{T}_{U}\mathcal{B}X + \omega\mathcal{V}\nabla_{U}\mathcal{B}X + \mathcal{C}\mathcal{H}\nabla_{U}\mathcal{C}X\nonumber\\&& + \omega\mathcal{T}_{U}\mathcal{C}X - \mathcal{C}\mathcal{P}_{U}X - \omega\mathcal{Q}_{U}X\big\}.
\end{eqnarray*}
Hence, $(\nabla F_{*})(U, X) = 0$, if and only if, $(\mathcal{T}_{U}\mathcal{B}X + \mathcal{H}\nabla_{U}\mathcal{C}X - \mathcal{P}_{U}X)\in\Gamma(\omega\mathcal{D}^{\bot})$ and $(\mathcal{V}\nabla_{U}\mathcal{B}X + \mathcal{T}_{U}\mathcal{C}X - \mathcal{Q}_{U}X)\in\Gamma(\mathcal{D})$. Thus, the proof is complete.
\end{proof}
\begin{corollary}
Using (\ref{eq:5}) and the symmetry of the second fundamental form of $F$, it follows that
\begin{eqnarray}\label{eq:21}
&g_{B}((\nabla F_{*})(U, V), F_{*}X)& = - g_{B}(F_{*}\nabla_{U}V, F_{*}X) = - g_{B}(F_{*}\mathcal{H}\nabla_{U}V, F_{*}X)\nonumber\\&&
= - g_{M}(\mathcal{T}_{U}V, X),
\end{eqnarray}
and
\begin{eqnarray}\label{eq:22}
&g_{B}((\nabla F_{*})(U, X), F_{*}Y)& = - g_{B}(F_{*}\nabla_{X}U, F_{*}Y) = - g_{B}(F_{*}\mathcal{H}\nabla_{X}U, F_{*}Y)\nonumber\\&&
= - g_{M}(\mathcal{A}_{X}U, Y) = g_{M}(\mathcal{A}_{X}Y, U).
\end{eqnarray}
This implies that $F$ is a totally geodesic map if and only if $\mathcal{T}_{U}V = 0$ and $\mathcal{A}_{X}Y = 0$, for any $U, V\in\Gamma(kerF_{*})$ and $X, Y\in\Gamma(kerF_{*})^{\bot}$.
\end{corollary}
\begin{theorem}
Let $F:(M, g_{M}, J)\rightarrow (B, g_{B})$ be a generic Riemannian submersion from a nearly Kaehler manifold onto a Riemannian manifold. Then $\phi$ is parallel with respect to $\nabla$ if and only if $\mathcal{Q}_{U}V = \mathcal{T}_{U}\omega V - \mathcal{B}\mathcal{T}_{U}V$, for any $U, V\in\Gamma(ker F_{*})$.
\end{theorem}
\begin{proof}
For $U, V, W\in\Gamma(ker F_{*})$ and $X\in\Gamma(ker F_{*})^{\bot}$, it is easy to see that $g_{M}(\mathcal{T}_{U}X, V) = - g_{M}(X, \mathcal{T}_{U}V)$ then from (\ref{eq:11c}), it follows that
$$
g_{M}((\nabla_{U}\phi)W, V) =  g_{M}(J\mathcal{T}_{U}W, V) + g_{M}(\omega W, \mathcal{T}_{U}V) + g_{M}(\mathcal{Q}_{U}W, V).
$$
Since $g_{M}(\mathcal{Q}_{U}W, V) = - g_{M}(W, \mathcal{Q}_{U}V)$ then last expression becomes
$$
g_{M}((\nabla_{U}\phi)W, V) =  g_{M}(W, \mathcal{T}_{U}\omega V) - g_{M}(W, \mathcal{B}\mathcal{T}_{U}V) - g_{M}(W, \mathcal{Q}_{U}V),
$$
this completes the proof.
\end{proof}
\begin{theorem} Let $F:(M, g_{M}, J)\rightarrow (B, g_{B})$ be a generic Riemannian submersion from a nearly Kaehler manifold onto a Riemannian manifold. Then $\omega$ is parallel with respect to $\nabla$ if and only if $\mathcal{P}_{U}X = \mathcal{T}_{U}\mathcal{C}X - \phi\mathcal{T}_{U}X$, for any $U \in\Gamma(kerF_{*})$ and $X \in\Gamma(kerF_{*})^{\bot}$.
\end{theorem}
\begin{proof}
Let $U, V \in\Gamma(kerF_{*})$ and $X \in\Gamma(kerF_{*})^{\bot}$ then from (\ref{eq:11a}), we have
$$
g_{M}((\nabla_{U}\omega)V, X) = - g_{M}(\mathcal{T}_{U}V, \mathcal{C}X) + g_{M}(\phi V, \mathcal{T}_{U}X) + g_{M}(\mathcal{P}_{U}V, X).
$$
This further implies
$$
g_{M}((\nabla_{U}\omega)V, X) = g_{M}(V, \mathcal{T}_{U}\mathcal{C}X) - g_{M}(V, \phi\mathcal{T}_{U}X) - g_{M}(V, \mathcal{P}_{U}X),
$$
this completes the proof.
\end{proof}
Analogously, we can derive the following assertions.
\begin{theorem} Let $F:(M, g_{M}, J)\rightarrow (B, g_{B})$ be a generic Riemannian submersion from a nearly Kaehler manifold onto a Riemannian manifold. Then $\mathcal{B}$ is parallel with respect to $\nabla$ if and only if $\mathcal{Q}_{U}V = \mathcal{T}_{U}\phi V - \mathcal{C}\mathcal{T}_{U}V$, for any $U, V \in\Gamma(kerF_{*})$.
\end{theorem}
\begin{theorem} Let $F:(M, g_{M}, J)\rightarrow (B, g_{B})$ be a generic Riemannian submersion from a nearly Kaehler manifold onto a Riemannian manifold. Then $\mathcal{C}$ is parallel with respect to $\nabla$ if and only if $\mathcal{P}_{U}X = \mathcal{T}_{U}\mathcal{B}X - \omega\mathcal{T}_{U}X$, for any $U \in\Gamma(kerF_{*})$ and $X \in\Gamma(kerF_{*})^{\bot}$.
\end{theorem}
\begin{theorem} Let $F:(M, g_{M}, J)\rightarrow (B, g_{B})$ be a generic Riemannian submersion from a nearly Kaehler manifold onto a Riemannian manifold. If $\omega$ is parallel with respect to $\nabla$ then
\begin{eqnarray}\label{eq:3.10}
\mathcal{T}_{\phi U} {\phi U} = - \mathcal{T}_{U}\mathcal{B}\omega U - \mathcal{T}_{U}U - 2 \mathcal{P}_{U}\phi U,
\end{eqnarray}
for any $U \in\Gamma(kerF_{*})$.
\end{theorem}
\begin{proof} Let $\omega$ be parallel with respect to $\nabla$ then from (\ref{eq:11a}), we have
$$
\mathcal{C}\mathcal{T}_{U}V  + \mathcal{P}_{U}V = \mathcal{T}_{U}\phi V,
$$
for any $U, V\in\Gamma(kerF_{*})$. Interchange the role of $U$, $V$ and then subtract the resulting equation from the last equation and further on using (\ref{eq:3}), we obtain
$$
2\mathcal{P}_{U} V= \mathcal{T}_{U}\phi V - \mathcal{T}_{V} \phi U.
$$
Furthermore, on substituting $V$ as $\phi U$ and then using the Lemma \ref{lem:4} (ii), the assertion follows.
\end{proof}
Now, we recall that a Riemannian submersion $F: (M, g_{M})\rightarrow (B, g_{B})$ between Riemannian manifolds is called a Riemannian submersion with totally umbilical fibers if
\begin{equation}\label{eq:23}
\mathcal{T}_{U}V = H^{\star}g_{M}(U, V),
\end{equation}
for $U, V\in\Gamma(ker F_{*})$, where $H^{\star}$ is the mean curvature vector of the fibers.
\begin{theorem}
Let $F$ be a generic Riemannian submersion with totally umbilical fibers from a nearly Kaehler manifold $(M, g, J)$ onto a Riemannian manifold $(B, g_{B})$. Then $H^{\star}\in\Gamma(\omega\mathcal{D}^{\bot})$.
\end{theorem}
\begin{proof}
Let $U, V\in\Gamma(\mathcal{D})$ then using the Lemma~\ref{lem:2} with (\ref{eq:1}), (\ref{eq:8}) and (\ref{eq:9}), it follows that
$$
\mathcal{T}_{U}\phi V + \mathcal{V}\nabla_{U}\phi V = \mathcal{B}\mathcal{T}_{U}V + \mathcal{C}\mathcal{T}_{U}V + \phi\mathcal{V}\nabla_{U}V + \omega\mathcal{V}\nabla_{U}V - \mathcal{P}_{V}U - \mathcal{Q}_{V}U.
$$
Take scalar product of last expression with $X\in\Gamma(\mu)$ and using (\ref{eq:23}), we get
\begin{equation}\label{eq:24}
g_{M}(U, \phi V)g_{M}(H^{\star}, X) = - g_{M}(U, V)g_{M}(H^{\star}, \mathcal{C}X) + g_{M}(\mathcal{P}_{U}V, X).
\end{equation}
Interchange the role of $U$, $V$ and then on adding the resulting expression with (\ref{eq:24}), it follows that $g_{M}(U, V)g_{M}(H^{\star}, \mathcal{C}X) = 0$, further, the non-degeneracy of $\Gamma(\mathcal{D})$ and $\Gamma(\mu)$ implies that $H^{\star}\in\Gamma(\omega\mathcal{D}^{\bot})$.
\end{proof}
\begin{theorem}
Let $F:(M, g_{M}, J)\rightarrow (B, g_{B})$ be a generic Riemannian submersion from a nearly Kaehler manifold onto a Riemannian manifold with totally umbilical fibers such that $\mathcal{P}_{U}\phi U = 0$, for any $U\in\Gamma(ker F_{*})$. If $\omega$ is parallel with respect to $\nabla$ then $F$ is with totally geodesic fibers.
\end{theorem}
\begin{proof}
For any $U\in\Gamma(\mathcal{D})$, from (\ref{eq:1}), it follows that
$$
(\nabla_{\phi U}J)U + (\nabla_{U}J)\phi U = 0.
$$
Further, from the Lemma~\ref{lem:2} with (\ref{eq:8})--(\ref{eq:10}) and the hypothesis $\mathcal{P}_{U}\phi U = 0$, we obtain
\begin{equation*}
\mathcal{T}_{\phi U}\phi U + \mathcal{V}\nabla_{\phi U}\phi U - \mathcal{B}\mathcal{T}_{\phi U}U - \mathcal{C}\mathcal{T}_{\phi U}U - \phi\mathcal{V}\nabla_{\phi U}U - \omega\mathcal{V}\nabla_{\phi U}U + \mathcal{Q}_{U}\phi U = 0.
\end{equation*}
Since $\omega$ is parallel with respect to $\nabla$ then for any $U\in\Gamma(\mathcal{D})$, from (\ref{eq:11b}), we have $\omega\mathcal{V}\nabla_{\phi U}U = 0$ and using (\ref{eq:3.10}) in the last expression, we get
\begin{equation*}
- \mathcal{T}_{U}U + \mathcal{V}\nabla_{\phi U}\phi U - \mathcal{B}\mathcal{T}_{\phi U}U - \mathcal{C}\mathcal{T}_{\phi U}U - \phi\mathcal{V}\nabla_{\phi U}U + \mathcal{Q}_{U}\phi U = 0.
\end{equation*}
On taking scalar product of the last expression with $X\in\Gamma(\omega(\mathcal{D}^{\bot}))$, we get $g_{M}(\mathcal{T}_{U}U, X) = 0$. Fibers of the submersion $F$ are totally umbilical therefore, we have $g_{M}(U, U)g_{M}(H^{\star}, X) = 0$, then non-degeneracy of $(ker F_{*})$ and $\omega(\mathcal{D}^{\bot})$ gives $H^{\star} = 0$, this completes the proof.
\end{proof}
\begin{theorem}
Let $F:(M, g_{M}, J)\rightarrow (B, g_{B})$ be a generic Riemannian submersion from a nearly Kaehler manifold onto a Riemannian manifold. Then $M$ is a locally twisted product manifold of the form $M_{(ker F_{*})^{\bot}}\times_{f}M_{(ker F_{*})}$ if and only if
\begin{equation*}
\mathcal{A}_{X}\mathcal{B}Y + \mathcal{H}\nabla_{X}\mathcal{C}Y - \mathcal{P}_{X}Y\in\Gamma(\mu),\quad \mathcal{V}\nabla_{X}\mathcal{B}Y + \mathcal{A}_{X}\mathcal{C}Y - \mathcal{Q}_{X}Y = 0,
\end{equation*}
$$
\mathcal{T}_{U}JX = - g_{M}(X, \mathcal{T}_{U}U)\|U\|^{- 2}JU - J\nabla_{U}X,
$$
for any $U, V\in\Gamma(ker F_{*})$ and $X\in\Gamma(ker F_{*})^{\bot}$.
\end{theorem}
\begin{proof}
Let $U, V\in\Gamma(ker F_{*})$ and $X\in\Gamma(ker F_{*})^{\bot}$ then from (\ref{eq:1}), we have
$$
g_{M}(\nabla_{U}V, X) = - g_{M}(JV, \nabla_{U}JX - (\nabla_{U}J)X).
$$
Further, from the Lemma~\ref{lem:2}, we get
$$
g_{M}(\nabla_{U}V, X) = - g_{M}(JV, \mathcal{T}_{U}\mathcal{B}X + \mathcal{V}\nabla_{U}\mathcal{B}X + \mathcal{H}\nabla_{U}\mathcal{C}X + \mathcal{T}_{U}\mathcal{C}X - \mathcal{P}_{U}X - \mathcal{Q}_{U}X),
$$
this implies that $(ker F_{*})$ is totally umbilical if and only if
$$
\mathcal{T}_{U}\mathcal{B}X + \mathcal{V}\nabla_{U}\mathcal{B}X + \mathcal{H}\nabla_{U}\mathcal{C}X + \mathcal{T}_{U}\mathcal{C}X - \mathcal{P}_{U}X - \mathcal{Q}_{U}X = - X(\lambda)JU,
$$
where $\lambda$ is some function on $M$. Then by straightforward calculations, we obtain $X(\lambda) = g_{M}(X, \mathcal{T}_{U}U)\|U\|^{- 2}$ and hence
$$
\mathcal{T}_{U}JX = - g_{M}(X, \mathcal{T}_{U}U)\|U\|^{- 2}JU - J\nabla_{U}X.
$$
Thus, the proof follows form the Theorem~\ref{thm:A} and Theorem~\ref{thm:4}.
\end{proof}
It is well known that if a nearly Kaehler manifold $M$ is of constant holomorphic sectional curvature $c(m)$ at every point $m\in M$ then the Riemannian curvature tensor of $M$ is of the form \cite{ss}
\begin{eqnarray}\label{eq:25}
R(X, Y, Z, W)& =& \frac{c(m)}{4}\Big\{g_{M}(X, W)g_{M}(Y, Z) - g_{M}(X, Z)g_{M}(Y, W)\nonumber\\&&
+ g_{M}(X, JW)g_{M}(Y, JZ) - g_{M}(X, JZ)g_{M}(Y, JW)\nonumber\\&&
- 2 g_{M}(X, JY)g_{M}(Z, JW)\Big\}\nonumber\\&&
+ \frac{1}{4}\Big\{g_{M}((\nabla_{X}J)W, (\nabla_{Y}J)Z) - g_{M}((\nabla_{X}J)Z, (\nabla_{Y}J)W)\nonumber\\&&
- 2g_{M}((\nabla_{X}J)Y, (\nabla_{Z}J)W)\Big\},
\end{eqnarray}
for any vector fields $X, Y, Z, W$ on $M$.\\
\indent Now, we recall following result from \cite{Gray70} for later use.
\begin{theorem}
Let $M$ be a nearly Kaehler manifold. Then $M$ has (pointwise) constant type if and only if there exists a smooth function $\alpha$ on $M$ such that
\begin{equation}\label{eq:26}
\|(\nabla_{X}J)Y\|^{2} = \alpha\big\{\|X\|^{2}\|Y\|^{2} - g_{M}(X, Y)^{2} - g_{M}(X, JY)^{2}\big\},
\end{equation}
for any vector fields $X, Y$ on $M$. Furthermore, $M$ has global constant type if and only if (\ref{eq:26}) holds with a constant function $\alpha$. In this case $\alpha$ is called the constant type of $M$.
\end{theorem}
\begin{theorem}
Let $F:(M, g_{M}, J)\rightarrow (B, g_{B})$ be a generic Riemannian submersion from a nearly Kaehler manifold onto a Riemannian manifold such that $M$ is with constant holomorphic sectional curvature $c$ and with constant type $\alpha$. If the distribution $\mathcal{D}$ is integrable and the submersion $F$ is with totally umbilical fibers then $c = \alpha$.
\end{theorem}
\begin{proof}
For a nearly Kaehler manifold $M$, it is known that $(\nabla_{U}J)(JV) = - J((\nabla_{U}J)V)$, for any vector fields $U$ and $V$ on $M$. Using this fact for $U\in\Gamma(\mathcal{D})$ and $V\in\Gamma(\mathcal{D}^{\bot})$ with (\ref{eq:25}), we derive
\begin{equation}\label{eq:27}
g_{M}(R(U, \phi U)V, \omega V) = - \frac{c}{2}g_{M}(U, U)g_{M}(V, V) + \frac{1}{2}\|(\nabla_{U}J)V\|^{2}.
\end{equation}
From \cite{neill}, it is known that
\begin{equation}\label{eq:28}
g_{M}(R(U, V)W, X) = g_{M}((\nabla_{V}\mathcal{T})_{U}W, X) - g_{M}((\nabla_{U}\mathcal{T})_{V}W, X),
\end{equation}
for any $U, V, W\in\Gamma(ker F_{*})$ and $X\in\Gamma(ker F_{*})^{\bot}$. Hence, from (\ref{eq:28}), we have
\begin{eqnarray}\label{eq:29}
g_{M}(R(U, \phi U)V, \omega V)& =& g_{M}(\nabla_{\phi U}\mathcal{T}_{U}V - \mathcal{T}_{\nabla_{\phi U}U}V - \mathcal{T}_{U}\nabla_{\phi U}V, \omega V)\nonumber\\&&
- g_{M}(\nabla_{U}\mathcal{T}_{\phi U}V - \mathcal{T}_{\nabla_{U}\phi U}V - \mathcal{T}_{\phi U}\nabla_{U}V, \omega V).
\end{eqnarray}
Since submersion $F$ is with totally umbilical fibers therefore $\mathcal{T}_{U}V = 0$, $\mathcal{T}_{\phi U}V = 0$ and using the fact the $\mathcal{T}$ is symmetric, (\ref{eq:29}) becomes
\begin{eqnarray}
g_{M}(R(U, \phi U)V, \omega V)& =& g_{M}(\mathcal{T}_{[U, \phi U]}V, \omega V) - g_{M}(\nabla_{U}V, \mathcal{T}_{\phi U}\omega V)\nonumber\\&& + g_{M}(\nabla_{\phi U}V, \mathcal{T}_{U}\omega V).\nonumber
\end{eqnarray}
Again using the hypothesis that the submersion $F$ is with totally umbilical fibers therefore $\mathcal{T}_{\phi U}\omega V = 0$ and $\mathcal{T}_{U}\omega V = 0$, therefore last expression becomes $g_{M}(R(U, \phi U)V, \omega V) = g_{M}([U, \phi U], V)g_{M}(H^{\star}, \omega V)$, as the distribution $\mathcal{D}$ is integrable, then we obtain
\begin{eqnarray}\label{eq:30}
g_{M}(R(U, \phi U)V, \omega V) = 0.
\end{eqnarray}
Hence from (\ref{eq:27}) and (\ref{eq:30}), we have $cg_{M}(U, U)g_{M}(V, V) = \|(\nabla_{U}J)V\|^{2}$, thus on using (\ref{eq:26}), proof is complete.
\end{proof}
\begin{theorem}
Let $F:(M, g_{M}, J)\rightarrow (B, g_{B})$ be a generic Riemannian submersion from a nearly Kaehler manifold onto a Riemannian manifold. Then $F$ is a harmonic map if and only if
$$
\nabla_{e_{i}}e_{i}^{\star} = \nabla_{e_{i}^{\star}}e_{i},\quad
(\mathcal{T}_{E_{j}}\phi E_{j} + \mathcal{H}\nabla_{E_{j}}\omega E_{j})\in\Gamma(\omega\mathcal{D}^{\bot}),
$$
and
$$
(\mathcal{V}\nabla_{E_{j}}\phi E_{j} + \mathcal{T}_{E_{j}}\omega E_{j})\in\Gamma(\mathcal{D}),
$$
for $\{e_{1},\ldots, e_{r}, e_{1}^{\star},\ldots, e_{r}^{\star}\}$ and $\{E_{1},\ldots, E_{k}\}$ orthogonal bases of $\mathcal{D}$ and $\mathcal{D}^{\bot}$, respectively.
\end{theorem}
\begin{proof}
It is known that the distribution $\mathcal{D}$ is $\phi-$invariant therefore take $\{e_{1},\ldots, e_{r}, e_{1}^{\star},\ldots, e_{r}^{\star}\}$ as an orthogonal basis of $\mathcal{D}$, where $e_{i}^{\star} = Je_{i} = \phi e_{i}$, for $i\in\{1,\ldots, r\}$ and let $\{E_{1},\ldots, E_{k}\}$ be an orthogonal basis of $\mathcal{D}^{\bot}$. Moreover, the second fundamental form $\nabla F_{*}$ of the Riemannian submersion $F$ satisfies $(\nabla F_{*})(X, Y) = 0$, for any $X, Y\in\Gamma(kerF_{*})^{\bot}$. Hence, the tension field $\tau(F)$ of $F$ is given by
$$
\tau(F) = \sum_{i = 1}^{r}\big\{(\nabla F_{*})(e_{i}, e_{i}) + (\nabla F_{*})(e_{i}^{\star}, e_{i}^{\star})\big\} + \sum_{j = 1}^{k}(\nabla F_{*})(E_{j}, E_{j}),
$$
using (\ref{eq:20a}), we further derive
\begin{eqnarray}
&\tau(F)& = \sum_{i = 1}^{r}\big\{F_{*}(\mathcal{C}\mathcal{T}_{e_{i}}e_{i}^{\star} + \omega\mathcal{V}\nabla_{e_{i}}e_{i}^{\star}) - F_{*}(\mathcal{C}\mathcal{T}_{e_{i}^{\star}}e_{i} + \omega\mathcal{V}\nabla_{e_{i}^{\star}}e_{i})\big\}\nonumber\\&&\quad
+ \sum_{j = 1}^{k}F_{*}(\mathcal{C}\mathcal{T}_{E_{j}}\phi E_{j} + \omega\mathcal{V}\nabla_{E_{j}}\phi E_{j} + \mathcal{C}\mathcal{H}\nabla_{E_{j}}\omega E_{j} + \omega\mathcal{T}_{E_{j}}\omega E_{j})\nonumber\\&&
= \sum_{i = 1}^{r}F_{*}(\omega\mathcal{V}(\nabla_{e_{i}}e_{i}^{\star} - \nabla_{e_{i}^{\star}}e_{i}))
+ \sum_{j = 1}^{k}F_{*}(\mathcal{C}(\mathcal{T}_{E_{j}}\phi E_{j} + \mathcal{H}\nabla_{E_{j}}\omega E_{j}))\nonumber\\&&\quad
+  \sum_{j = 1}^{k}F_{*}(\omega(\mathcal{V}\nabla_{E_{j}}\phi E_{j} + \mathcal{T}_{E_{j}}\omega E_{j})).
\end{eqnarray}
Thus, the proof is complete.
\end{proof}

\textbf{Author's Addresses:}\\
Rupali Kaushal and R. K. Nagaich\\
Department of Mathematics, Punjabi University, Patiala-147 002, India\\
Email: rupalimaths@pbi.ac.in; nagaich58rakesh@gmail.com\\[0.2cm]
Rashmi Sachdeva and Rakesh Kumar\\
Department of Basic and Applied Sciences, Punjabi University, Patiala-147 002, India\\
Email: rashmi.sachdeva86@gmail.com; dr$\_$rk37c@yahoo.co.in
\end{document}